\documentclass[12pt]{elsarticle}

\usepackage{amssymb}
\usepackage{amsthm}
\usepackage{amsmath}

\newtheorem{theorem}{Theorem}
\newtheorem{example}[theorem]{Example}

\begin{document}

\begin{frontmatter}



\title{ An explicit univariate and radical parametrization of the sextic proper Zolotarev polynomials}


\author[1]{Heinz-Joachim Rack}
\author[2]{Robert Vajda\corref{cor1} \fnref{fn1}}
\cortext[cor1]{Corresponding author}
\fntext[fn1]{Supported by the projects NKFI KH 125628 and 20391-3/2018/FEKUSTRAT UNKP.}

\address[1]{Steubenstrasse 26 a, 58097 Hagen, Germany}
\address[2]{Bolyai Institute, University of Szeged, Aradi Vertanuk tere 1, 6720 Szeged, Hungary}

\begin{abstract}
The problem to determine an explicit one-parameter power form representation 
of the proper Zolotarev polynomials of degree $n$ and with uniform norm $1$
on $[-1,1]$ can be traced back to P. L. Chebyshev.
It turned out to be complicated, even for 
small values of $n$. Such a representation was known to A. A. Markov (1889) for $n=2$ and $n=3$. \emph{But already
for $n=4$ it seems that nobody really believed that an explicit form can be found. 
As a matter of fact it was, by V. A. Markov in 1892}, as A. Shadrin put it in 2004. 
About 125 years passed before an explicit form for the next higher degree, $n=5$, was found, by G. Grasegger and N. Th. Vo (2017). 
In this paper we settle the case $n=6$.
\end{abstract}


\end{frontmatter}

\section{Introduction and historical remarks}\label{sectintro}

Chebyshev's extremal problem (CEP) of 1854 \cite{Chebyshev} is to determine among all monic polynomials of fixed degree $n\ge1$, given by
\begin{equation}\label{degngeneral} 
\tilde{P}_{n}(x)=\sum_{k=0}^{n-1} a_{k,n}x^{k}+x^n,
\end{equation}
where $a_{k,n}\in{\mathbb R}$ are arbitrary coefficients (and $a_{n,n}=1$), 
that particular one which deviates least from the zero-function on $\boldsymbol{I}=[-1,1]\subset{\mathbb R}$ measured in the uniform norm $||.||_{\infty}$.
Chebyshev found that the solution is given on $\boldsymbol{I}$ as follows:
\begin{equation}\label{Chebn}
\tilde{T}_n(x)=2^{1-n}T_n(x)=\sum_{k=0}^{n-1} a^{*}_{k,n}x^{k}+x^n=2^{1-n}\cos(n \arccos(x)),
\end{equation}
with least deviation $2^{1-n}$, known optimal coefficients $a^{*}_{k,n}$, and $T_n$ with $||T_n||_{\infty}=1$ denoting
the $n$-th Chebyshev polynomial of the first kind with respect to $\boldsymbol{I}$, 
see \cite[p. 384]{Milovanovic} or
 \cite[p. 6, p. 67]{Rivlin74} for details.

In 1867 Chebyshev himself proposed to his student E. I. Zolotarev, 
see \cite[p. 2]{Zolotarev1877}, an extension of CEP by requiring that not only the first but also the second leading coefficient,
$a_{n-1,n}$, is to be kept fixed. This extended CEP, which was later renamed as Zolotarev's first problem (ZFP), can be stated as follows: 

\emph{To determine among all monic polynomials of fixed degree $n\ge2$, represented as}
\begin{equation}\label{degnmonic}
\tilde{P}_{n,s}(x)=\sum_{k=0}^{n-2} a_{k,n}x^{k}+(-n s)x^{n-1}+x^n
\end{equation}
\emph{where $s\in\mathbb{R}\backslash\{0\}$ is prescribed,
that particular one, call it $\tilde{Z}_{n,s}$, with}
\begin{equation}\label{dnZol1}
\tilde{Z}_{n,s}(x)=\sum_{k=0}^{n-2} a^{*}_{k,n}(s)x^{k}+(-n s)x^{n-1}+x^n,
\end{equation}
\emph{which deviates least from the zero-function on $\boldsymbol{I}$ in the uniform norm $||.||_{\infty}$}.

Or put alternatively, the goal is to determine the best uniform approximation on $\boldsymbol{I}$ to $f(x,s)=(-n s)x^{n-1}+x^n$ 
by polynomials of degree $<n-1$.

It is well-known that one may restrict the parameter $s$ to $s>0$, and that for $0<s\le\tan^2\left(\pi/(2n)\right)$
the solution $\tilde{Z}_{n,s}$ is given by a distorted Chebyshev polynomial (see e.g. \cite[p. 16]{Achieser98}, \cite[p. 57]{Achieser03}, \cite{Carlson83}, 
\cite[p. 405]{Milovanovic} for details), 
and is called an \emph{improper monic Zolotarev polynomial}. 

However, for $s>\tan^2\left(\pi/(2n)\right)$,  the solution $\tilde{Z}_{n,s}$ to ZFP 
is considered as \emph{very complicated} (see e.g. \cite{Carlson83}, \cite[p. 407]{Milovanovic}, \cite{Peherstorfer91}) 
or even as \emph{mysterious} \cite{Todd84}, and is called \emph{a proper} \cite[p. 160]{Voronovskaja70}, or \emph{hard-core} \cite{Rivlin75} 
\emph{monic Zolotarev polynomial}. Here we shall consider only the cases $s>\tan^2(\pi/(2n))$, noting that
$0<\tan^2\left(\pi/(2n)\right)<1$ holds for $n>2$.
They find application (after rescaling) e.g. in the proof of the Markov inequality \cite{AAMarkov1889}, \cite{VAMarkov} and
Landau-Kolmogorov inequality \cite{Shadrin14},
in the proof of Schur's Markov-type problem \cite{Erdos42}, \cite{Rack17a}, \cite{Rack17b} and 
in the problem of maximizing linear coefficient functionals \cite{Rack89}.

Zolotarev provided a solution to ZFP in 1868 \cite{Zolotarev1868}, and in a reworked form in 1877 \cite{Zolotarev1877}, 
where he was considering altogether four extremal problems, of which ZFP was the first in the row (hence the name). 
Surprisingly, Zolotarev presented the proper monic $\tilde{Z}_{n,s}$ in terms of elliptic functions 
(see e.g. \cite[p. 18]{Achieser98}, \cite[p. 280]{Achieser03}, \cite{Carlson83}, \cite{Erdos42}, \cite[p. 407]{Milovanovic},
\cite{Peherstorfer04})
rather than, as is suggested by the task, in the power form (\ref{dnZol1}) with optimal coefficients $a^{*}_{k,n}(s)$. When compared to the two-fold solution (\ref{Chebn}) of CEP, Zolotarev's \emph{unwieldy} \cite[p. 118]{Tikhomirov} elliptic 
(or transcendental) solution of ZFP would correspond to the trigonometric right-hand solution in (\ref{Chebn}) 
without providing an equivalent algebraic left-hand term, see also \cite[p. 38]{Goncharov}.
The following statement by A. A. Markov \cite[p. 264]{AAMarkov1906} 
indicates a reservation about Zolotarev's elliptic solution: \emph{Being based on the application of elliptic functions, 
Zolotarev's solution is too complicated to be useful in practice}.

It is tempting to derive an explicit algebraic solution for the proper $\tilde{Z}_{n,s}$ from the elliptic solution. 
However, even for the first reasonable polynomial degree $n=2$ this path turns out be unexpectedly complicated, see \cite{Carlson83} 
for details. 
Therefore, alternative solution paths have been pursued to determine the proper $\tilde{Z}_{n,s}$. 
For example, \emph{A. A. Markov himself tried to employ the theory of continued fractions in order to find an algebraic solution [to ZFP],
but he was not fully successful, because an algebraic solution requires an amazing amount of calculations}, as is remarked in 
\cite[p. 932]{Malyshev02}. 

In 2004 Shadrin \cite{Shadrin04} wrote: \emph{Recently, the interest in an explicit algebraic solution of ZFP was revived in the papers
 Malyshev {\rm \cite{Malyshev03}}, Peherstorfer {\rm \cite{Peherstorfer90}}, Sodin-Yuditskii {\rm \cite{Yuditskii95}}, but it is only Malyshev who demonstrates how his theory can be applied to some explicit constructions 
for particular $n$.} 
But actually Malyshev \cite{Malyshev03}, see also \cite{Malyshev02}, provided explicit constructions only for
$2\le n\le 5$. Inspired by \cite{Malyshev03} we have provided in a recent paper \cite{Rack18} explicit 
algebraic solutions to ZFP for $6\le n\le11$ in terms of roots of dedicated polynomials by modifying results 
from \cite{Schiefermayr07} and utilizing computer algebra methods which are implemented in the software  
\emph{Mathematica\texttrademark} \cite{MMA1}.
The provision of a solution to ZFP for $n>5$ via computer algebra had been stated as an open problem in \cite{Kaltofen}.
Based on an advanced computer algebra strategy, in the conference paper \cite{Lazard} 
it is claimed to have algebraically solved ZFP even for $6\le n\le 12$. But we do not share this view, since the theoretical strategy 
in \cite{Lazard} appears not granulated finely enough for the purpose of enabling the construction of
$\tilde{Z}_{n,s}$ for a given $n$ and $s$, the more so as neither concrete examples nor a solution formula are provided. 
But we leave it to the reader to form an opinion. 

The mentioned algebraic solutions to ZFP do not meet the demand, which has been vibrant from the outset, 
for a description of the solution to ZFP which avoids elliptic functions and is represented as in (\ref{dnZol1}) 
analytically and explicitly in a power form (with coefficients which depend on a single parameter).
In answering the open problem which we have addressed in \cite[Remark 7]{Rack18} we are now able to show
that for $n=6$ such a parametrization of the coefficients of proper Zolotarev polynomials is in fact
a radical (and not a rational) one. We are going to provide it explicitly in Section \ref{sec6} below.

\section{Explicit analytical one-parameter power form representation of the normalized proper 
   Zolotarev polynomials of degree $n\le5$}\label{sec25}

If $L=L(n,s)>0$ denotes the deviation from zero of $\tilde{Z}_{n,s}$ on $\boldsymbol I$ (which is minimal compared to all polynomials of form 
(\ref{degnmonic})), then the scaled proper Zolotarev polynomial $\tilde{Z}_{n,s}/L$ clearly has uniform norm $1$ on $\boldsymbol I$. 
Proper Zolotarev polynomials with uniform norm $1$ will be called \emph{normalized}.

Such polynomials of degree $2\le n\le4$ and represented in a power form 
are scattered in the literature, see \cite{Carlson83}, \cite{Collins96}, \cite{Grasegger17},\cite[p. 156]{Paszkowski62}, 
\cite{Rack89}, \cite{Rack17a}, \cite{Rack17b}, \cite{Shadrin04} and \cite[p. 98]{Voronovskaja70}
(the latter with respect to $[0,1]$). 
They can be expressed (possibly after some rearrangements) analytically as 
\begin{equation}\label{parformnn}
Z_{n,t}(x)=\sum_{k=0}^{n} b_{k,n}(t) x^k, \quad {\rm with}\quad 0\neq b_{n,n}(t) \quad {\rm and\,} \quad t\in I_n,
\end{equation}
where the explicit coefficients $b_{k,n}(t)$ depend on a parameter $t$, and $I_n$ 
denotes a dedicated (finite, if $n>2$) open parameter interval.
As is addressed in the Abstract, the cases $n=2$ and $n=3$ are contained already in A. A. Markov (1889) \cite{AAMarkov1889}.
The case $n=4$, 
appearing in the form (\ref{parformnn}) in \cite{Grasegger17}, \cite{Rack89} and in \cite{Shadrin04}, 
deserves special attention: Shadrin \cite[p. 10]{Shadrin04} attributes it, 
see the Abstract, to V. A. Markov (1892) \cite{VAMarkov} 
(more precisely, as communicated privately to the first-named author, to a passage on p. 73 in \cite{VAMarkov} 
which is not contained in the abridged German translation \cite{VAMarkovG} of \cite{VAMarkov}, see also \cite[p. 160]{Rack17a}).  
Shadrin refers two times to the fact that representations (\ref{parformnn}) are available only for three values of $n$: 
\emph{explicit expressions... are known only for $n=2,3,4$} \cite[p. 10]{Shadrin04} 
and \emph{there is no explicit expression for [normalized proper] Zolotarev polynomials of degree $n>4$} \cite[p. 1185]{Shadrin14}.

It took about 125 years before a normalized proper Zolotarev polynomial of
the next higher degree, $n=5$, had been found in the desired form (\ref{parformnn}), see Grasegger and Vo (2017) \cite{Grasegger17}. 
Partial results for $n=5$ appeared earlier in \cite{Collins96} (for a correction see \cite[p. 73]{Rack17b}) 
and in \cite[p. 937]{Malyshev02}. 
In the next Section, we are going to reveal the case $n=6$.

For the sake of definiteness we shall assume, without loss of generality, 
that a normalized proper Zolotarev polynomial $Z_{n,t}$ in the form (\ref{parformnn}) satisfies certain definite conditions 
which follow from its intrinsic properties, see e.g. \cite{Achieser98}, \cite{Achieser03}:  $Z_{n,t}$ must equioscillate
$n$  times on $\boldsymbol{I}$ and, additionally, two times on some interval $[\alpha,\beta]$. 
We assume here that the following holds: $Z_{n,t}(-1)=(-1)^n, Z_{n,t}(1)=-1$, $Z_{n,t}(\alpha)=-1$, $Z_{n,t}(\beta)=1$,
where $1<\alpha<\beta$ and $||Z_{n,t}||_{\infty}=1$ for $x\in \boldsymbol I$ and $x\in[\alpha,\beta]$. 
In literature both $Z_{n,t}$ and $-Z_{n,t}$ are considered interchangeably as normalized proper Zolotarev polynomials. 
Less frequently the two polynomials defined by $\pm Z_{n,t}(-x)$ go by this name, in which case the two additional equioscillation points 
would be situated to the left of $\boldsymbol{I}$.

To deduce, for a given $s>\tan^2(\pi/(2n))$, from (\ref{parformnn}) the monic proper Zolotarev polynomial $\tilde{Z}_{n,s}$, 
one may proceed as follows: Divide (\ref{parformnn}) by $b_{n,n}(t)$ yielding
$\sum_{k=0}^{n-1} c_{k,n}(t)x^k+x^n$, then equate $c_{n-1,n}(t)$ with $(-ns)$ and solve for $t$, and finally insert the solution
$t=t^{*} \in I_n$ into $\sum_{k=0}^{n-1} c_{k,n}(t) x^k+x^n$ to get $\tilde{Z}_{n,s}$, 
see also \cite[Theorem 3]{Erdos42}. 
An example of such a deduction, for $n=5$ and $s=2$, is given in \cite[Section 5]{Rack17b}.
In anticipation of a result of the next Section, we mention that for $n=6$ there is exactly one instance where a normalized proper Zolotarev 
polynomial of form (\ref{parformnn}) is already monic: if $L=1$ holds, and this is the case if $t=-0.0003253\dots$, see Formula (\ref{Lt}) below.

\section{Explicit analytical one-parameter power form representation of the normalized proper 
    Zolotarev polynomials of degree $n=6$}\label{sec6}

Our main result is the representation of the family of normalized proper Zolotarev polynomials of degree $6$ 
in the parameterized power form (\ref{parformnn}). 
The parametrization for the cases $2\le n\le 4$ is a rational one, see \cite{Grasegger17} and \cite{Rack17b}, whereas for the case $n=5$ it is a radical one, 
and it also turns out to be so for the case $n=6$, see the even-indexed coefficients in Theorem \ref{mainthm} below.
We have achieved our result by using symbolic computation (Quantifier Elimination, Cylindrical Algebraic Decomposition, Groebner Basis) 
as implemented in \emph{Mathematica} and by using the Algebraic Curve Package \emph{algcurves} in \emph{Maple\texttrademark} \cite{maple}. 
However, it would be too bulky to reproduce here all the computational steps of our proof, which is similar to, but more complex, 
than the proof for $n=5$ in \cite{Grasegger17}. 
Therefore, we proceed as in \cite[Section 5]{Carlson83}: 
We give a proof in the nature of a verification, that is, we write down the sought-for family of polynomials in the one-parameter 
power form (\ref{parformnn}) and then we verify that they are indeed (sextic) normalized proper Zolotarev polynomials by checking 
that they satisfy the defining properties of such polynomials, see e.g. \cite{Achieser98}, \cite{Achieser03},
\cite{Carlson83}, \cite{Erdos42}, \cite{Peherstorfer99}, \cite{Shadrin04}. 
In particular, these properties are: existence of 6 equioscillation points on $\boldsymbol{I}$ (including the endpoints), 
existence of 3 points $\gamma<\alpha<\beta$ to the right of $\boldsymbol{I}$ where at $\gamma$ 
the first derivative vanishes 
and where $\alpha$ and $\beta$ are two further equioscillation points, 
solution of the Abel-Pell differential equation, 
solution of the Peherstorfer-Schiefermayr nonlinear system of equations, 
coincidence (for $n=6$) with known general limiting values when the parameter $t$ 
tends to the boundaries of the parameter interval. 

\begin{theorem}\label{mainthm}
Let $t$ denote a real parameter from the finite open parameter interval
\begin{equation}\label{inti6}
I_6=\left(\frac{1}{2}(5-3\sqrt3),0\right), \,{\rm with\,\,}  \frac{1}{2}(5-3\sqrt3)=-0.09807\dots\,,
\end{equation} and let $\omega=\omega(t)$ denote the radical expression $\sqrt{(-1+t)t(1+t+7t^2)}$.
For every $t\in I_6$ the sextic algebraic polynomial $Z_{6,t}$ in $x$, with
\begin{equation}\label{parform6}
Z_{6,t}(x)=\sum_{k=0}^{6} b_{k,6}(t)x^k,
\end{equation} 
is a normalized proper Zolotarev polynomial of degree $n=6$ on $\boldsymbol{I}$. The parameterized coefficients $b_{k,6}(t)$ are given by
\begin{flalign}\label{b06}
b_{0,6}(t)=&\frac{2 \sqrt{3} (-1 + t)^2 \omega}{(1 + 2 t)^5 (-1 + 4 t)^3 (1 - 2 t + 10 t^2)^4}\times&&\\
&\big(1 - 6 t + 18 t^2 - 16 t^3 - 252 t^4 + 2592 t^5 - 5844 t^6 +  20448 t^7 -&&\nonumber\\ 
&15768 t^8 - 219280 t^9 + 942576 t^{10} - 893232 t^{11} + 2825968 t^{12}\big)\nonumber&&
\end{flalign}
\begin{flalign}\label{b16}
b_{1,6}(t)=&
 \frac{(-5+6t-24t^2-4t^3)}{(1-4t)^2(1+2t)^5(1-2t+10t^2)^4} \times&&\\
&\big(1 - 12 t^2 + 116 t^3 - 756 t^4 + 2520 t^5 + 1212 t^6 - 12744 t^7 +&&\nonumber\\ 
&  69840 t^8 - 309280 t^9 + 700704 t^{10} - 709008 t^{11} + 788848 t^{12}\big)\nonumber&&
\end{flalign}
\begin{flalign}\label{b26}
b_{2,6}(t)=&\frac{2 \sqrt{3} (-1 + t)^2\omega}{(1 + 2 t)^5 (1 - 4 t)^3 (1 - 2 t + 10 t^2)^4}\times&&\\
&\big(13 - 102 t + 390 t^2 - 880 t^3 - 288 t^4 + 19296 t^5 -102792 t^6 +&&\nonumber\\ 
&  390816 t^7 - 939024 t^8 + 1167536 t^9 - 258720 t^{10} - 339888 t^{11} +&&\nonumber\\ 
&  2720848 t^{12}\big)\nonumber&&
\end{flalign}
\begin{flalign}\label{b36}
b_{3,6}(t)=&
\frac{-4(-1+t)^5}{(1-4t)^2(1+6t^2+20t^3)^4}\times&&\\
&\big(5 + 3 t - 6 t^2 + 564 t^3 - 3408 t^4 + 13296 t^5 - 35136 t^6 +&&\nonumber\\
&  
 107976 t^7 - 130416 t^8 + 243952 t^9\big)\nonumber&&
\end{flalign}
\begin{flalign}\label{b46}
b_{4,6}(t)=&\frac{8 \sqrt{3} (1 - t)^7 \omega}{(1 + 2 t)^5 (-1 + 4 t)^3 (1 - 2 t + 10 t^2)^4}\times&&\\
& \big(7 - 25 t + 66 t^2 - 146 t^3 - 64 t^4 + 2580 t^5 - 6800 t^6 + 26252 t^7\big) \nonumber&&
\end{flalign}
\begin{flalign}\label{b56}
b_{5,6}(t)=
&\frac{-16(-1 + t)^{10} (1 + t + 7 t^2) (1 + 6 t + 12 t^2 + 116 t^3)}{(1 + 2 t)^5 (1- 4 t)^2 (1 - 2 t + 10 t^2)^4}&&
\end{flalign}
\begin{flalign}\label{b66} 
b_{6,6}(t)=
&\frac{-32 \sqrt{3} (-1 + t)^{12}  (1 + t + 7 t^2) \omega}{(1 + 2 t)^5 (-1 + 4 t)^3 (1 - 2 t + 10 t^2)^4}.&&
\end{flalign}
We note that $b_{0,6}(t)=-(b_{2,6}(t)+b_{4,6}(t)+b_{6,6}(t))$ and $b_{1,6}(t)=-(1+b_{3,6}(t)+b_{5,6}(t))$ holds.

The connection to $\tilde{Z}_{n,s}$, the monic proper Zolotarev polynomial of degree $n=6$, see {\rm (\ref{dnZol1})}, 
is established via the equation
\begin{flalign}\label{st}
s=s(t)=\frac{(1 - 4 t) (1 + 6 t + 12 t^2 + 116 t^3)\omega}{12 \sqrt{3} (-1 + t)^3 t (1 + t + 7 t^2)}&&
\end{flalign}
and via the representation of the (least) deviation of $\tilde{Z}_{n,s}$ from zero on $\boldsymbol I$,
\begin{flalign}\label{Lt}
L=L(6,s)=L(t)=\frac{(1 - 4 t)^3 (1 + 2 t)^5 (1 - 2 t + 10 t^2)^4 \omega}{32 \sqrt{3} (-1 + t)^{13} t (1 + t + 7 t^2)^2}.&&
\end{flalign}\qed
\end{theorem}
\begin{proof}
One may first verify that $Z_{6,t}$ and its first derivative $Z'_{6,t}$ attain dedicated values $y\in\{-1,0,1\}$ at selected points
$x\in \{-1, 1, \alpha, \beta, \gamma, z_1, z_2, z_3, z_4\}$, due to the intrinsic structure of normalized proper Zolotarev polynomials:
\begin{equation}\label{equioscpoints}
Z_{6,t}(-1)=1,\, Z_{6,t}(1)=-1, \, Z'_{6,t}(\gamma)=0, \, Z_{6,t}(\alpha)=-1,\, Z_{6,t}(\beta)=1,
\end{equation}
where
\begin{equation}\label{gammat}
\gamma=\gamma(t)=\frac{(1-4t)(5-6t+24t^2+4t^3)}{12\sqrt{3}(-1+t)^2\omega}
\end{equation}
\begin{equation}\label{alphat}
\alpha=\alpha(t)=\frac{-9t^2}{(-1+t)^2}+\frac{(1+2t)(1-4t)(1-2t+10t^2)}{2\sqrt{3}(-1+t)^2\omega}
\end{equation}
\begin{equation}\label{betat}
\begin{split}
\beta=&\beta(t)=\frac{9t^2}{(-1+t)^2}+\frac{(1+2t)(1-4t)(1-2t+10t^2)}{2\sqrt{3}(-1+t)^2\omega}\\
=&\frac{18t^2}{(-1+t)^2}+\alpha.
\end{split}
\end{equation}
We note that $\gamma=(\alpha+\beta)/2-s$ holds, see \cite[p. 2486]{Yuditskii95}.
Denote the $4$ inner equioscillation points of $Z_{6,t}$ on $\boldsymbol{I}$ as  $z_1<z_2<z_3<z_4$.
One may then verify that they are given, together with the associated values of $Z_{6,t}$ and of $Z'_{6,t}$, as follows:
\begin{equation}\label{z1t}
z_1=z_1(t)=A-B {\rm \,\,with\,\,} Z_{6,t}(z_1)=-1 {\rm \,\,and\,\,} Z_{6,t}'(z_1)=0, {\,\,\rm where}
\end{equation}
\begin{equation}
A=A(t)=\frac{(-1+4t)((1+2t)-\frac{\sqrt{3}}{\omega} t (1 + t + 16 t^2))}{4(-1+t)^2},
\end{equation}
\begin{equation}
\begin{split}
B=B(t)=&\frac{(1+2t)}{4(-1+t)^2}\times\\
& \sqrt{\frac{2\sqrt{3}\omega (1+2t)(-1+4t)+(5-26t+102 t^2 - 200 t^3 + 524 t^4)}{1+t+7t^2}};
\end{split}
\end{equation}
\begin{equation}\label{z2t}
z_2=z_2(t)=C-D {\rm \,\,with\,\,} Z_{6,t}(z_2)=1 {\rm \,\,and\,\,} Z_{6,t}'(z_2)=0, {\,\,\rm where}
\end{equation}
\begin{equation}
C=C(t)=\frac{(1-4t)((1+2t)+\frac{\sqrt{3}}{\omega} t(1+t+16t^2))}{4(-1+t)^2},
\end{equation}
\begin{equation}
\begin{split}
D=D(t)=&\frac{(1+2t)}{4(-1+t)^2}\times\\
& \sqrt{\frac{-2\sqrt{3}\omega (1 + 2 t) (-1 + 4 t)+(5-26t+102 t^2 - 200 t^3 + 524 t^4)}{1+t+7t^2}};
\end{split}
\end{equation}
\begin{equation}\label{z3t}
z_3=z_3(t)=A+B {\rm \,\,with\,\,} Z_{6,t}(z_3)=-1 {\rm \,\,and\,\,} Z_{6,t}'(z_3)=0;
\end{equation}
\begin{equation}\label{z4t}
z_4=z_4(t)=C+D {\rm \,\,with\,\,} Z_{6,t}(z_4)=1 {\rm \,\,and\,\,} Z_{6,t}'(z_4)=0.
\end{equation}
One may furthermore verify that the polynomial $Z_{6,t}$ satisfies the Abel-Pell differential equation, which for $n=6$ reads, 
see e.g. \cite[p. 17]{Achieser98}, \cite{Bogatyrev}, \cite[p. 10]{Shadrin04},
\begin{equation}\label{DE6a}
\frac{(1-x^2)(x-\alpha)(x-\beta) (Z_{6,t}'(x))^2}{36 (x-\gamma)^2} =1-(Z_{6,t}(x))^2. 
\end{equation}
Next, one may verify that the equioscillation points of the polynomial $Z_{6,t}$ satisfy the Peherstorfer-Schiefermayr system of nonlinear equations 
which, for $n=6$, reads, see \cite[Lemma 2.1 and p. 68]{Peherstorfer99}, \cite[Lemma 1]{Schiefermayr07}:
\begin{equation}
\alpha+\beta+2(z_1+z_2+z_3+z_4)-\frac{(1-4t)(1+6t+12t^2+116t^3)}{\sqrt{3}(-1+t)^2\omega}=0.
\end{equation}
\begin{equation}
-1+(-1)^k+2(-z_1^k+z_2^k-z_3^k+z_4^k)-\alpha^k+\beta^k=0 \,\,{\rm for\,\,} k=1,2,3,4,5.
\end{equation}
Its validity implies two alternative representations of $Z_{6,t}$, see \cite[p. 150]{Schiefermayr07}: 
\begin{align}
Z_{6,t}(x)=& 1-\frac{2(x+1)(x-\beta)(x-z_2)^2(x-z_4)^2}{(\alpha+1)(\alpha-\beta)(\alpha-z_2)^2(\alpha-z_4)^2}\\
=&-1+\frac{(x-\alpha)(x-1)(x-z_1)^2(x-z_3)^2}{(1+\alpha)(1+z_1)^2(1+z_3)^2}.\label{twoalternative}
\end{align}
We note that the denominator in (\ref{twoalternative}) can be rewritten as $(\beta-1)(\beta-\alpha)(\beta-z_1)^2(\beta-z_3)^2/2$.

Finally, one may verify that the limiting behavior of $Z_{6,t}$ when $t$ tends towards $0$ and towards $(5-3\sqrt3)/2$ is, see 
\cite[p. 19]{Achieser98} and \cite[pp. 247-248]{Lebedev94}:
\begin{equation}\label{lim0}
\lim_{t\to0} Z_{6,t}(x)=-T_5(x),\quad {\rm where\,\,} T_5(x)=5x-20x^3+16x^5 
\end{equation}
and
\begin{equation}\label{limt0}
\lim_{t\to\frac{1}{2}(5-3\sqrt3)} Z_{6,t}(x)=T_6\left(\frac{(x+1)(2+\sqrt3)}{4}-1\right),
\end{equation}
where $T_6(x)=-1+18x^2-48x^4+32x^6$, 
and the limiting behavior of $\alpha$ and $\beta$ is:
\begin{equation}
\lim_{t\to0}\alpha(t)=\infty,\, \lim_{t\to\frac{1}{2}(5-3\sqrt{3})}\alpha(t)=1,
\end{equation}
\begin{equation}
\lim_{t\to0}\beta(t)\!=\!\infty,\, \lim_{t\to\frac{1}{2}(5-3\sqrt{3})}\beta(t)\!=\!1+2\tan^2\left(\frac{\pi}{12}\right)\!=\!15-8\sqrt3\!=\!1.14359\dots,
\end{equation}
in accordance with \cite[p. 454]{Erdos42}.

In order to deduce, for $n=6$ and for a given $s>\tan^2(\pi/12)$, from (\ref{parform6}) the monic proper Zolotarev polynomial $\tilde{Z}_{n,s}$
and thus to solve ZFP 
(see the final paragraph of Section \ref{sec25}), we divide (\ref{parform6}) by $b_{6,6}(t)$ so that the first leading coefficient turns into $1$ 
and the second one turns into
\begin{equation}
\frac{b_{5,6}(t)}{b_{6,6}(t)}=\frac{(-1 + 4 t) (1 + 6 t + 12 t^2 + 116 t^3)}{2 \sqrt{3} (-1 + t)^2 \omega}.
\end{equation}
Identifying this term with $-6s$ yields that $s$ is the term as given in {\rm (\ref{st})}. Evaluating $\tilde{Z}_{6,s}$ at $x=-1$ 
yields that the (minimal) deviation $L=L(6,s)$ is the term as given in {\rm (\ref{Lt})}.

All these verifications we have accomplished with the aid of \emph{Mathematica} and have cross-checked the results 
with \emph{Maple}. We leave it to the reader to reverify the above properties with a method of own choice. 
\end{proof}

\begin{example}\label{exarat}
Choosing $t=-1/20=-0.05 \in I_6$ yields 
\begin{align}
&\gamma=\frac{3176}{147\sqrt{301}}=1.24531\dots,\\ 
&\alpha=\frac{-301+1200\sqrt{301}}{14749}=1.39116\dots,\\
&\beta=\frac{301+1200\sqrt{301}}{14749}=1.43197\dots\,
\end{align}
and
\begin{align}
&z_1=\frac{-3612-88\sqrt{301}-21\sqrt{43(3251-24\sqrt{301})}}{14749}=-0.84550\dots,\\ 
&z_2=\frac{3612-88\sqrt{301}-21\sqrt{43(3251+24\sqrt{301})}}{14749}=-0.42403\dots,\\
&z_3=\frac{-3612-88\sqrt{301}+21\sqrt{43(3251-24\sqrt{301})}}{14749}=0.14868\dots,\\
&z_4=\frac{3612-88\sqrt{301}+21\sqrt{43(3251+24\sqrt{301})}}{14749}=0.70680\dots\,.
\end{align}
The corresponding sextic normalized proper Zolotarev polynomial is
\begin{equation}\label{Zolexarat}
\begin{split}
&Z_{6,t=-0.05}(x)=\frac{1}{777600000000}\times\\
&\big(
 -31735420507 \sqrt{301} - 2906886359536 x + 
  452607070657 \sqrt{301} x^2 +\\
& 12429463839072 x^3 -
  1016046999793 \sqrt{301} x^4 -
  10300177479536 x^5 + \\
&  595175349643 \sqrt{301} x^6\big)=\\
& (-0.70806\dots) +(- 3.73827\dots) x +
 (10.09830\dots) x^2 +\\
& (15.98439\dots) x^3 + 
 (- 22.66944\dots) x^4 +(- 13.24611\dots) x^5 +\\
& (13.27920\dots) x^6.
\end{split}
\end{equation}
It is readily seen that it satisfies, for example, the conditions {\rm (\ref{equioscpoints})}, {\rm (\ref{z1t})}, {\rm (\ref{z2t})},
{\rm (\ref{z3t})}, {\rm (\ref{z4t})}.
The graph of $Z_{6,t=-0.05}$, whose uniform norm on $\boldsymbol{I}$ and on $[\alpha, \beta]$ is $1$, is displayed in Figure \ref{figure 1},
where the two vertical lines indicate the interval $[\alpha,\beta]$. \qed
\begin{figure}[h]
\begin{center}
\includegraphics[width=4.0in]{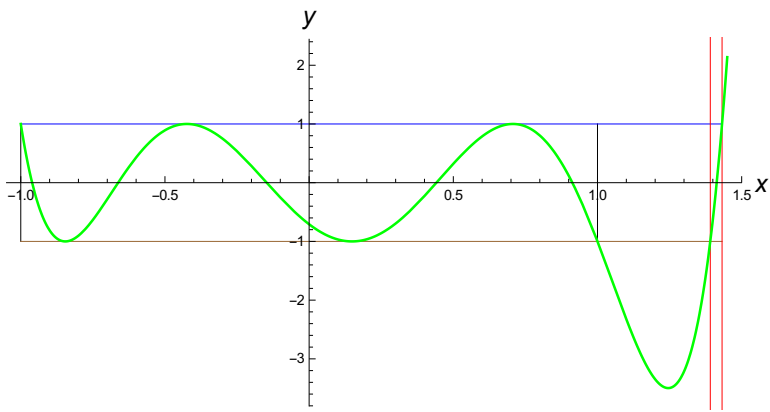}
\end{center}
\caption{$Z_{6,t=-0.05}$}\label{figure 1}
\end{figure}
\end{example}

\begin{example}\label{exas1}
The goal is  to solve ZFP for $n=6$ and, say, $s=1>\tan^2(\pi/12)$ $=7-4\sqrt3=0.07179\dots\,\,$.
To this end, solve equation {\rm(15)} with $s=1$ for the variable $t$ and choose the unique solution
$t=t^{*}=-0.002272...\in I_6$, which is a root of the polynomial 
$1 + 436 x - 1748 x^2 + 5272 x^3 - 15632 x^4 + 24592 x^5 - 12752 x^6 - 
 48416 x^7 + 212272 x^8$. Then insert $t^{*}$ into $Z_{6,t}/b_{6,6}(t)$ in order to get the desired
solution to ZFP, see {\rm (4)}:
\begin{equation}
\begin{split}
&\tilde{Z}_{6,s=1}(x)=(-0.06207\dots)+(-1.86731\dots)x+(0.81036\dots)x^2+\\
&(7.48972\dots)x^3+(-1.74828\dots)x^4+(-6)x^5+x^6.
\end{split}
\end{equation}
The least deviation from zero is $\tilde{Z}_{6,s=1}(-1)=L=L(6,s=1)=-\tilde{Z}_{6,s=1}(1)=0.37758\dots\,\,$.
This solution to ZFP for $n=6$ and $s=1$ coincides with the one which was determined independently in {\rm\cite[Example 2]{Rack18}}.
 \qed
\end{example}

\section{Concluding remarks}

\subsection{The rational side-solution of the sextic Abel-Pell differential equation}
Regrettably, we have to point to a flaw in the paper by Grasegger and 
Vo \cite{Grasegger17} concerning the degree $n=6$: 
The one-parameter power form representation as given there in Section 4.5, and identically given in Section 4.6 (Example 4.1), 
expressed there as $T_3(Z_2(x))$, which is in fact a rational solution of the sextic Abel-Pell differential equation (\ref{DE6a}), does not represent, as is claimed in \cite{Grasegger17}, 
a family of sextic normalized proper Zolotarev polynomials. The reason is that for each parameter $t>1$
the sextic polynomial $T_3(Z_2(x))$ equioscillates less than six times (in fact four times) on $\boldsymbol{I}$. 
Here, $T_3(x)=-3x+4x^3$ and $Z_2(x)=(1+2 t x-x^2)/t^2$ with $t>1$ so that
\begin{equation}
\begin{split}
T_3(Z_2(x))=&-\frac{1}{2t^3} \big((-1+3t^2)+(-6t+6t^3)x+(3-15t^2)x^2+\\
&(12t-8t^3)x^3+(-3+12t^2)x^4+(-6t)x^5+x^6\big).
\end{split}
\end{equation}
Observe that $Z_{2}$ with $t>1$ denotes here (in our notation) the family $-Z_{2,t}$ of negative normalized proper 
Zolotarev polynomials of degree $n=2$, satisfying $-Z_{2,t}(-1)=-1$, see \cite[pp. 2-3]{Carlson83}.
Thus we have to contrast $-T_3(Z_2(x))$ with our solution $Z_{6,t}(x)$ as given in (\ref{parform6}), whereof the disparity 
becomes obvious immediately.
The gap in the proof of Corollary 4.3 in \cite{Grasegger17} is the omission of the check whether the considered polynomials 
equioscillate on $\boldsymbol{I}$ exactly as many times as their degree indicates. 
An underlying fault is a misinterpretation of a result of Lebedev \cite{Lebedev94} which enters into 
Theorem 4.2 in \cite{Grasegger17}. This item has already been pointed to in \cite[Remark 9]{Rack17b}. 

\subsection{Asymptotics for the least deviation $L$}

S. N. Bernstein \cite{Bernstein1913} provided for $n\to\infty$ the following asymptotic approximation, $L_{\infty}$,
to the constant $L$ in (\ref{Lt}):
\begin{equation}\label{Linf}
L_{\infty}=\frac{n s+\sqrt{n^2 s^2+1}}{2^{n-1}}. 
\end{equation}
Already for $n=6$ this approximation is quite formidable as can be concluded from our examples.
 
In Example \ref{exarat} we have $t=-1/20$ and hence by (\ref{Lt}) we get $L=$\\ $\frac{777600000000}{595175349643\sqrt{301}}=0.07530\dots$ (which is the inverse of the leading coefficient 
in (\ref{Zolexarat})).
On the other hand, with $s=s(-1/20)=\frac{424}{147\sqrt{301}}=0.16625\dots$ according to (\ref{st}), we get from (\ref{Linf}) that
$L_{\infty}=\frac{848+\sqrt{1441805}}{1568\sqrt{301}}=0.07531\dots$ holds.

In Example \ref{exas1}, where $s=1$ holds, we have obtained $L=0.37758\dots\,\,$. From (\ref{Linf}) we get 
$L_{\infty}=\frac{1}{32}(6+\sqrt{37})=0.37758\dots\,\,$.  
Using higher precision one sees that this is a match in ten digits after decimal point.

\subsection{Choice of the parameter interval}
In our search for a convenient finite parameter interval, we have stopped after having found, in January 2019, $I_6$ 
as given in (\ref{inti6}), since it resembles $I_5=(\frac{1}{5}(-5+2\sqrt{5}),0)$ as given in \cite[p. 178]{Grasegger17}.
In the mean time we have gotten the hint that simplifications in our above formulas for $\alpha, \beta, \gamma$ 
can be achieved when  $t$ will be replaced by a certain rational transformation of $t\in I_6$.
But we retain here our primal choice $I_6$.


%




\begin{thebibliography}{00}




\bibitem{Achieser98} 
{\sc N.I. Achieser}, {\it Function theory according to Chebyshev}. In: Mathematics of the 
     19th century, Vol. 3 (A.N. Kolmogorov et al. (Eds.)), Birkh\"auser, Basel, 1998, 
     pp. 1-81 (Russian 1987).

\bibitem{Achieser03} 
{\sc N.I. Achieser}, {\it Theory of Approximation}, Dover Publications, Mineola (NY), USA, 2003 (Russian 1947). DOI: 10.2307/3612294 

\bibitem{Bernstein1913}
{\sc S.N. Bernstein}, {\it Sur quelques propri\'et\'es asymptotiques des polynomes}, C. R. 
     {\bf 157} (1913), pp. 1055-1057. 





\bibitem{Bogatyrev}
{\sc A. Bogatyrev}, {\it Extremal Polynomials and Riemann Surfaces}, Springer, New York (NY), USA, 2012 (Russian 2005),
DOI: 10.1007/978-3-642-25634-9

\bibitem{Carlson83} 
{\sc B.C. Carlson, J. Todd}, {\it Zolotarev's first problem - the best approximation by polynomials 
of degree $\le n-2$ to $x^n-n\sigma x^{n-1}$} in $[-1,1]$, Aeq. Math. {\bf 26} (1983), pp. 1-33. DOI: 10.1007/bf02189661

\bibitem{Chebyshev} 
{\sc P.L. Chebyshev}, {\it Th\'eorie des m\'ecanismes connus sous le nom de  
     parall\'elogrammes}, Mem. Acad. Sci. St. Petersburg {\bf 7} (1854), pp. 539-568. 
URL {\tt http://www.math.technion.ac.il/hat/fpapers/cheb11.pdf}


\bibitem{Collins96} 
{\sc G.E. Collins}, {\it Application of quantifier elimination to Solotareff's approximation problem}. 
In: Stability Theory (Hurwitz Centenary Conference, Ascona, Switzerland, 1995, R. Jeltsch et al. (Eds.)), Birkh\"auser, Basel, 
ISNM {\bf 121} (1996), pp. 181-190. DOI: 10.1007/978-3-0348-9208-7\_19

\bibitem{Erdos42}
{\sc P. Erd\H{o}s, G. Szeg\H{o}}, {\it On a problem of I. Schur}, Ann. Math. {\bf 43} (1942), pp. 451-470. DOI: 10.2307/1968803 





\bibitem{Goncharov}
{\sc V.L. Goncharov}, {\it The theory of best approximation of functions}, J. Approx. Theory {\bf 106} (2000), pp. 2-57 (Russian 1945).
DOI: 10.1006/jath.2000.3476

\bibitem{Grasegger17}
{\sc G. Grasegger, N.Th. Vo}, {\it An algebraic-geometric method for computing  
       Zolotarev polynomials}.
In: Proceedings International Symposium on Symbolic and 
       Algebraic Computation (ISSAC '17, Kaiserslautern, Germany, M. Burr (Ed.)), 
       ACM, New York, 2017, pp. 173-180. DOI: 10.1145/3087604.3087613


\bibitem{Kaltofen} 
{\sc E. Kaltofen}, {\it Challenges of symbolic computation: My favorite open problems}, 
       J. Symb. Comput. {\bf 29} (2000), pp. 891-919. DOI: 10.1006/jsco.2000.0370

\bibitem{Lazard} 
{\sc D. Lazard}, {\it Solving Kaltofen's challenge on Zolotarev's approximation}.
In: Proceedings International Symposium on Symbolic and Algebraic Computation 
       (ISSAC '06, Genoa, Italy, B. Trager (Ed.)), ACM, New York, 2006, pp. 196-203. DOI: 10.1145/1145768.1145803

\bibitem{Lebedev94}
{\sc V.I. Lebedev}, {\it Zolotarev polynomials and extremum problems}, Russ. J. Numer. 
       Anal. Math. Model. {\bf 9} (1994), pp. 231-263. DOI: 10.1515/rnam.1994.9.3.231 


\bibitem{Malyshev02} 
{\sc V.A. Malyshev}, {\it The Abel equation}, St. Petersburg Math. J. {\bf 13} (2002), pp. 893-938 (Russian 2001).


\bibitem{Malyshev03} 
{\sc V.A. Malyshev}, {\it Algebraic solution of Zolotarev's problem}, St. Petersburg Math. 
       J. {\bf 14} (2003), pp. 711-712 (Russian 2002).

\bibitem{maple}
{\sc Maplesoft}, {\it Maple 2017}, Maplesoft, a division of Waterloo Maple Inc., Waterloo, Ontario.

\bibitem{AAMarkov1889}
{\sc A.A. Markov}, {\it On a question of D.I. Mendeleev} (Russian), Zapiski Imper. Akad. 
       Nauk St. Petersburg {\bf 62} (1889), pp. 1-24. URL {\tt http://www.math.technion.ac.il/hat/fpapers/mar1.pdf}
   

\bibitem{AAMarkov1906}
{\sc A.A. Markov}, {\it Lectures on functions of minimal deviation from zero} (Russian), 
       1906. In: Selected Works: Continued fractions and the theory of functions deviating 
       least from zero, OGIZ, Moscow-Leningrad, 1948, pp. 244-291.


\bibitem{VAMarkov}
{\sc V.A. Markov}, {\it On functions deviating least from zero on a given interval} 
       (Russian), Izdat. Akad. Nauk St. Petersburg 1892. 
URL {\tt  http://www.math.technion.ac.il/hat/fpapers/vmar.pdf}

\bibitem{VAMarkovG}
{\sc W. Markoff}, {\it \"Uber Polynome, die in einem gegebenen Intervalle m\"oglichst wenig von Null abweichen}, 
Math. Ann. {\bf 77} (1916), pp. 213-258.

\bibitem{Milovanovic}
{\sc G.V. Milovanovi\'c, D.S. Mitrinovi\'c, Th.M. Rassias}, {\it Topics in Polynomials: 
       Extremal Problems, Inequalities, Zeros}, World Scientific, Singapore, 1994. DOI: 10.1142/9789814360463

\bibitem{Paszkowski62}
{\sc S. Paszkowski}, {\it The Theory of Uniform Approximation I. Non-asymptotic 
       Theoretical Problems}, Rozprawy Matematyczne XXVI, PWN, Warsaw, 1962.


\bibitem{Peherstorfer90}
{\sc F. Peherstorfer}, {\it Orthogonal- and Chebyshev polynomials on two intervals}, 
       Acta Math. Hung. {\bf 55} (1990), pp. 245-278. DOI: 10.1007/bf01950935

\bibitem{Peherstorfer91}
{\sc F. Peherstorfer}, {\it On the connection of Posse's $L_1$- and Zolotarev's maximum-norm problem}, 
J. Approx. Theory {\bf 66} (1991), pp. 288-301. DOI: 10.1016/0021-9045(91)90032-6


\bibitem{Peherstorfer99} 
{\sc F. Peherstorfer, K. Schiefermayr}, {\it Description of extremal polynomials on 
       several intervals and their computation I, II}, Acta Math. Hungar. {\bf 83} (1999), pp. 27-58, pp. 59-83.
DOI: 10.1023/a:1006607401740 and DOI: 10.1023/a:1006659402649


\bibitem{Peherstorfer04} 
{\sc F. Peherstorfer, K. Schiefermayr}, {\it Description of inverse polynomial    
       images which consist of two Jordan arcs with the help of Jacobi's elliptic functions}, 
       Comput. Methods Funct. Theory {\bf 4} (2004), pp. 355-390. DOI: 10.1007/bf03321075




\bibitem{Rack89} 
{\sc H.-J. Rack}, {\it On polynomials with largest coefficient sums}, J. Approx. Theory {\bf 56} 
       (1989), pp. 348-359. DOI: 10.1016/0021-9045(89)90124-x

\bibitem{Rack17a} 
{\sc H.-J. Rack}, {\it The first Zolotarev case in the Erd\H{o}s-Szeg\H{o} solution to a Markov-type 
       extremal problem of Schur}, Stud. Univ. Babes-Bolyai Math. {\bf 62} (2017), pp. 151-162. DOI: 10.24193/subbmath.2017.2.02

\bibitem{Rack17b} 
{\sc H.-J. Rack}, {\it The second Zolotarev case in the Erd\H{o}s-Szeg\H{o} solution to a Markov-
       type extremal problem of Schur}, J. Numer. Anal. Approx. Theory {\bf 46} (2017), pp. 54-77.

\bibitem{Rack18}
{\sc H.-J. Rack, R. Vajda}, {\it Explicit algebraic solution of Zolotarev's first problem for low-degree polynomials}, submitted, 2018.




\bibitem{Rivlin74} 
{\sc Th.J. Rivlin}, {\it Chebyshev Polynomials}, 2nd edition, J. Wiley \& Sons, New York (NY), USA, 1990. DOI: 10.2307/2153227 

\bibitem{Rivlin75}
{\sc Th.J. Rivlin}, {\it Optimally stable Lagrangian numerical differentiation}, SIAM J. Numer. Anal. {\bf 12} (1975), pp. 712-725. 
DOI: 10.1137/0712053


\bibitem{Schiefermayr07} 
{\sc K. Schiefermayr}, {\it Inverse polynomial images which consists of two Jordan arcs 
       - An algebraic solution}, J. Approx. Theory {\bf 148} (2007), pp. 148-157. DOI: 10.1016/j.jat.2007.03.003


\bibitem{Shadrin04}
{\sc A. Shadrin}, {\it Twelve proofs of the Markov inequality}. 
In: Approximation Theory - A volume dedicated to B. Bojanov (D.K. Dimitrov et al. (Eds.)), M. Drinov Acad. 
       Publ. House, Sofia, 2004, pp. 233-298. 
URL {\tt http://www.damtp.cam.ac.uk/user/na/people/Alexei/papers/markov.pdf}

\bibitem{Shadrin14}
{\sc A. Shadrin}, {\it The Landau-Kolmogorov inequality revisited}, Discrete Contin. Dyn. 
       Syst. {\bf 34} (2014), pp. 1183-1210. DOI: 10.3934/dcds.2014.34.1183



\bibitem{Yuditskii95} 
{\sc M.L. Sodin, P.M. Yuditskii}, {\it Algebraic solution of a problem of E.I. Zolotarev 
       and N.I. Akhiezer on polynomials with smallest deviation from zero}, J. Math. Sci. {\bf 76} 
       (1995), pp. 2486-2492. DOI: 10.1007/bf02364906

\bibitem{Tikhomirov}
{\sc V.M. Tikhomirov}, {\it II. Approximation Theory}. In: Analysis II (R.V. Gramkelidze (Ed.)), Encyclopaedia 
of Mathematical Sciences, Vol. 14, Springer, New York (NY), USA, 1990, pp. 93-243 (Russian 1987). 
DOI: 10.1007/978-3-642-61267-1\_2 

\bibitem{Todd84}
{\sc J. Todd}, {\it Applications of transformation theory: A legacy from Zolotarev 
       (1847-1878)}, In: Approximation Theory and Spline Functions, Proceedings NATO 
       Advanced Study Institute (ASIC 136, St. John's, Canada, 1983, S.P. Singh et al. 
       (Eds.)), 1984, pp. 207-245. DOI: 10.1007/978-94-009-6466-2\_11 


\bibitem{Voronovskaja70} 
{\sc E.V. Voronovskaja}, {\it The Functional Method and its Applications}, Translations 
       of Mathematical Monographs, Vol. 28, AMS, Providence (RI), USA, 1970 (Russian 
       1963). DOI: 10.1090/mmono/028 


\bibitem{MMA1}
{\sc Wolfram Research, Inc.}, {\it Mathematica}, Version 11.0, Champaign (IL), USA, 2016.

\bibitem{Zolotarev1868}{\sc E.I. Zolotarev}, {\it On a problem of least values, Dissertation pro venia legendi}
  (Russian), lithographed paper, St. Petersburg (1868), Collected Works, Vol. 2, 
       Leningrad 1932, pp. 130-166.

\bibitem{Zolotarev1877} {\sc E.I. Zolotarev}, {\it Applications of elliptic functions to problems on functions 
       deviating least or most from zero} (Russian), Zapiski Imper. Akad. Nauk St. 
       Petersburg {\bf 30} (1877), Collected Works, Vol. 2, Leningrad, 1932, pp. 1-59. 
URL {\tt  http://www.math.technion.ac.il/hat/fpapers/zolo1.pdf}





\end{thebibliography}


\end{document}